\documentclass[12pt,compress]{article}
\usepackage{amsmath, amsthm}
\usepackage{amssymb}

\usepackage{mathrsfs}
\usepackage{scrextend}
\usepackage{enumerate}
\usepackage{longtable,booktabs,makecell,multirow,tabularx}
\usepackage{graphicx}

\usepackage{color}
\usepackage{hyperref}
\usepackage{geometry}
\newtheorem{theorem}{Theorem}[section]
\newtheorem{lemma}[theorem]{Lemma}

\newtheorem{theoremA}{Theorem}

\theoremstyle{definition}
\newtheorem{definition}[theorem]{Definition}
\newtheorem{remark}[theorem]{Remark}

\newtheorem{example}[theorem]{Example}

\numberwithin{equation}{section}

\def\UU{{\mathcal U}}

\def\LL{{\mathscr L}}

\usepackage[numbers,sort&compress]{natbib}

\usepackage{indentfirst}

\begin{document}


	\title{A note on $ p $-supersolubility of finite groups}
	
	\author{Zhengtian Qiu, Guiyun Chen and Jianjun Liu\footnote{Corresponding author.}\vspace{0.5em}\\
		{\small School of Mathematics and Statistics, Southwest University,}\\
		{\small  Chongqing 400715, People's Republic of China
		}\vspace{0.5em}\\
		{\small  E-mail addresses: qztqzt506@163.com, gychen1963@163.com, liujj198123@163.com}}
	
	\date{}
	\maketitle
	
	\begin{abstract}
	Let $ H $ be a subgroup of a finite group $ G $. We say that $ H $ satisfies  $ \mathscr L $-$ \Pi $-property in $ G $ if $ | G / K : N _{G / K} (HK/K)| $ is a $ \pi (HK/K) $-number for all maximal $ G $-invariant subgroup $ K $ of $ H^{G} $.  In this paper, we give a characterization of finite $p$-supersoluble groups under assumption that some subgroups of prime power order satisfy $ \mathscr L $-$ \Pi $-property.
	\end{abstract}

	\renewcommand{\thefootnote}{\fnsymbol{footnote}}
	
	
	
	\section{Introduction}

    All groups considered in this paper are finite.
    We use conventional notions as in \cite{MR1169099} or \cite{Gorenstein-1980}.

    Throughout the paper,  $ G $ is always a finite group and $ p $ is always a prime number. Let $ \pi $ denote some set of primes and $ \pi(G) $ denote the set of all prime divisors of $ |G| $. A chief factor factor $ L/K $ of $ G $ is said to be a \emph{$ p $-$ G $-chief factor}  if the prime $ p $ divides $ |L/K| $. We write ``$ H\lesssim G $" to mean that $ H $ is  isomorphic to a subgroup of $ G $.    A $ 2 $-group is called \emph{quaternion-free}
    if it has no section isomorphic to the quaternion group of order $ 8 $.

    Recall that the \emph{generalized Fitting subgroup}  $ F^{*}(G) $ of $ G $ is the unique maximal normal quasinilpotent subgroup of $ G $ (see \cite[Chapter X]{MR662826}), and the \emph{generalized $ p $-Fitting subgroup}  $ F_{p}^{*}(G) $ of $G$ is the normal subgroup of $ G $ such that $ F_{p}^{*}(G)/O_{p'}(G) = F^{*}(G/O_{p'}(G)) $ (see \cite{ball-2009}). The \emph{supersoluble hypercenter}  $ Z_{\UU}(G) $ of  $ G $ is the product of all normal subgroups $ H $ of $ G $, such that all $ G $-chief factors below $ H $ have prime order, and the \emph{$ p $-supersoluble hypercenter}  $ Z_{\UU_{p}}(G) $ of $G$ is the product of all normal subgroups $ H $ of $ G $, such that all $ p $-$ G $-chief factors below $ H $ have order $ p $ for some fixed prime $ p $.















    It is interesting to investigate the influence of some kind of embedding properties of subgroups on the structure of a finite group and many significant results have been given (for example, see \cite{Wang, Skiba-2007, Su-2014, Qian-2021}). In 2011, Li \cite{li-2011} introduced  the defnitions of the \emph{$ \Pi $-property} and \emph{$ \Pi $-normality} of subgroups of finite groups, which generalize a large number of other known embedding properties (see \cite[Section 2]{li-2011}). Let $ H $ be a subgroup of a group  $ G $.  $ H $ is said to  satisfy  $ \Pi  $-property in $ G $ if for any   chief factor $ L / K $ of $ G $, $ | G / K : N _{G / K} ( HK/K\cap L/K )| $ is a $ \pi (HK/K\cap L/K) $-number; and we call that $ H $ is $ \Pi $-normal in $ G $ if there exist a subnormal subgroup $ T $ of $ G $  such that $ G = HT $ and $ H \cap T \leq I \leq H $, where  $ I $  satisfies $ \Pi $-property in $ G $.


    In 2014, Su, Li and Wang \cite{Su-2014} gived the following criterion of a normal subgroup $ E $ of $ G $ being contained in $ Z_{\UU_{p}}(G) $ in terms of $ \Pi $-normality of some $ p $-subgroups of $ E $, which extended many  results.

    \begin{theorem}[{\cite[Main Theorem]{Su-2014}}]
    	Let $ p $ be a prime, $ E $ a normal subgroup of  $ G $ such that $ p $ divides the order of $ E $, and $ X $ a normal subgroup of $ G $ satisfying $ F_{p}^{*}(E) \leq X \leq E $. Then $ E\leq Z_{\UU_{p}}(G) $ if and only if a Sylow $ p $-subgroup $ P $ of $ X $ has a subgroup $ D $ such that $ 1 < | D | \leq \mathrm{max}(p, |P|/p ) $ and all subgroups $ H $ of $ P $ of order $ | H | = | D | $ and
    	all cyclic subgroups of $ P $ of order $ 4 $ (if $ P $ is a non-abelian $ 2 $-group and $ | D | = 2 $) are $ \Pi $-normal in $ G $.
    \end{theorem}

    From the definition of $ \Pi $-property of subgroups, all chief factors of $ G $  have been considered.  Now we introduce a new concept by considering fewer chief factors of $ G $, which generalizes the concept of $ \Pi $-property of subgroups.

    \begin{definition}
    	Let $ H $ be  a subgroup of $ G $, we say that $ H $ satisfies \emph{$ \LL $-$ \Pi $-property}  in $ G $ if $ | G / K : N _{G / K} (HK/K)| $ is a $ \pi (HK/K) $-number for any  $ G $-chief factor  of type $ H^{G}/K $, which means  $ K $ is a maximal $ G $-invariant subgroup  of $ H^{G} $.
    \end{definition}

    Let $ H $ be a subgroup of $ G $. If $ H^{G} $ has no maximal $ G $-invariant subgroups, we still say that $ H $ satisfies  $ \LL $-$ \Pi $-property in $ G $.

    It is easy to see that if a subgroup $ H $ of $ G $ satisfies  $ \Pi $-property in $ G $, then $ H $ satisfies $ \LL $-$ \Pi $-property in $ G $.  But the converse fails in general. Let us see the following  example.

    \begin{example}
    	Let $ G = S_{4} $ be the symmetric group of degree $ 4 $. It is easy to see that all subgroups  of $ G $ of order $ 4 $ satisfy $ \LL $-$ \Pi $-property in $ G $.  Observe that $ \langle (1234)\rangle $ does not satisfy  $ \Pi $-property in $ G $.
    \end{example}

    In this paper, we  obtain a criterion for the $ p $-supersolubility of a finite group  by  assuming that some $ p $-subgroups of fixed order satisfy $ \LL $-$ \Pi $-property. Our main result is  as follows.

    \begin{theoremA}\label{main}
    	Let $ P $ be a Sylow $ p $-subgroup of $ G $, and let $ d $ be a power of $ p $ such that $ 1 < d < |P| $. Assume that $ P $ satisfies the following:
    	
    	\begin{enumerate}[\rm(1)]
    		\item Every  subgroup of $ P $ of order $ d $ satisfies $ \LL $-$ \Pi $-property in $ G $;
    		
    		\item If $ d=p=2 $ and $ P $ is not quaternion-free, we  further suppose that every cyclic subgroup of $ P $ of order $ 4 $ satisfies $ \LL $-$ \Pi $-property in $ G $.
    		
        \end{enumerate}	
    	
    	\noindent Then $ G $ is $ p $-supersoluble  whenever $ d = p $, or $ d\leq  |P\cap O_{p'p}(G)|/p $, or $ d\leq \sqrt{|P|} $.
    	
    \end{theoremA}


    \begin{remark}
    	Let $ G = S_{4} $ be the symmetric group of degree $ 4 $. It is easy to see that every subgroup of $ G $ of order $ 4 $  satisfies  $ \LL $-$ \Pi $-property in $ G $. But  $ S_{4} $ is not $ 2 $-supersoluble. Hence  the additional hypothesis ``$ d = p $, or $ d\leq  |P\cap O_{p'p}(G)|/p $, or $ d\leq \sqrt{|P|} $" in Theorem \ref{main} is  essential.
    \end{remark}



\section{Preliminaries}

In this section, we give some lemmas that will be used in our proofs.

\begin{lemma}\label{ove}
	Let $ H $ be a subgroup of $ G $ and $ N\unlhd G $, we have:
	
	{\rm (1)} If $ H $ satisfies $ \LL $-$ \Pi $-property in $ G $, then $ HN $ satisfies  $ \LL $-$ \Pi $-property in $ G $;
	
	{\rm (2)} If $ H $ satisfies $ \LL $-$ \Pi $-property in $ G $, then $ HN/N $ satisfies $ \LL $-$ \Pi $-property in $ G/N $,
	
	{\rm (3)} Assume that $ N\leq H $ or $ (|H|,|N|) = 1 $. Then $ H $ satisfies  $ \LL $-$ \Pi $-property in $ G $ if and only if $ HN/N $ satisfies  $ \LL $-$ \Pi $-property in $ G/N $.
\end{lemma}
\begin{proof}[\bf{Proof}]
	(1)  Let $ K $ be a maximal $ G $-invariant subgroup  of $ (HN)^{G} $. Suppose that $ N\leq K $. Note that $ (HN)^{G}/K=H^{G}N/K=H^{G}K/K\cong H^{G}/(H^{G}\cap K) $ is a chief factor of $ G $. Since $ H $ satisfies $ \LL $-$ \Pi $-property in $ G $, we have that $ | G : N _{G} (H(H^{G}\cap K))| $ is a $ \pi (H(H^{G}\cap K)/(H^{G}\cap K)) $-number. Note that $ N_{G}(HK)=N_{G}((H^{G}\cap HK)K)\geq N_{G}(H^{G}\cap HK) $ and $ H(H^{G}\cap K)/(H^{G}\cap K)\cong H/(H\cap K)\cong HK/K $. We see that  $ | G  : N _{G} (HK)| $ is a $ \pi (HK/K) $-number. This meas that $ HN $ satisfies  $ \LL $-$ \Pi $-property in $ G $, as wanted. Suppose that $ N\nleq K $. Since $ (HN)^{G} /K $ is a chief factor of $ G $, we have $ (HN)^{G} = NK = HNK $, thus $ HN $ satisfies  $ \LL $-$ \Pi $-property in $ G $, as wanted.
	
	(2) We may assume $ N \leq H $ by (1). Let $ M/N $ be a maximal $ G/N $-invariant subgroup of $ (H/N)^{G/N} $. Clearly,  $ H^{G} /M $ is a chief factor of $ G $. Since $ H $ satisfies $ \LL $-$ \Pi $-property in $ G $, we have that  $ |G/M :N_{G/M} (HM/M)| $ is a $ \pi (HM/M) $-number. Write $ \overline{G}=G/N $. Then  $ (\overline{H})^{\overline{G}} /\overline{M} $ is a chief factor of $ \overline{G} $. Hence $ |\overline{G}/\overline{M} :N_{\overline{G}/\overline{M}} (\overline{HM}/\overline{M})| $ is a $ \pi (\overline{HM}/\overline{M}) $-number. This means that $ HN/N $ satisfies $ \LL $-$ \Pi $-property in $ G/N $.
	
	(3) We only need to prove the sufficiency. Assume that $ HN/N $ satisfies $ \LL $-$ \Pi $-property in $ G/N $ and let $ H^{G} /M $ be a chief factor of $ G $. Set $ \overline{G}=G/N $.
	
	Suppose that $ N \leq H $. If $ N\leq M $, then $ (\overline{H})^{\overline{G}} /\overline{M} $ is a chief factor of $ \overline{G} $. Hence  $ |\overline{G}/\overline{M} :N_{\overline{G}/\overline{M}} (\overline{HM}/\overline{M})| $ is a $ \pi (\overline{HM}/\overline{M}) $-number, and so $ |G :N_{G} (HM)| $ is a $ \pi (HM/M) $-number.
	This means that $ H $ satisfies $ \LL $-$ \Pi $-property in $ G $, as desired. If $ N\nleq M $, then $ H^{G}=MN $. It follows that  $ | G  : N _{G} (HMN)|=| G  : N _{G} (H^{G})|=1 $, and hence $ H $ satisfies  $ \LL $-$ \Pi $-property in $ G $.
	
	Suppose that $ (|H|,|N|) = 1 $. Let $ \pi=\pi(H) $.  Note that $ O^{\pi'}(H^{G}) = H^{G} $ and $$ (\overline{H})^{\overline{G}}/\overline{M}\cong H^{G}N/MN\cong H^{G}/( H^{G}\cap MN)=H^{G}/M(H^{G}\cap N). $$
	If $ H^{G}\cap N\nleq M $, then $ H^{G}=M(H^{G}\cap N) $, and so $ H^{G}N=MN $. This yields  that $ H^{G}/M \leq MN/M = N/(N\cap M) $
	is a $ \pi' $-group, a contradiction. Hence $ H^{G} \cap N \leq M $. Thus $ (\overline{H})^{\overline{G}}/\overline{M} $ is a chief factor of $ \overline{G} $. Then $ |\overline{G}:N_{\overline{G}}(\overline{H}\overline{M})| $ is a $ \pi(\overline{HM}/\overline{M}) $-number. Thus $ |G:N_{G}(HMN)| $ is a $ \pi(HMN/MN) $-number. Since $ H^{G}\cap N\leq M\leq HM $, we have $ H^{G}\cap N= M\cap N= HM\cap N $. Clearly, $ HMN\cap H^{G}=HM(N\cap H^{G})=HM(M\cap N)=HM $. Note that $ |\overline{HM}/\overline{M}|=|HM/(HM\cap N)||(M\cap N)/M|=|HM/M| $. Therefore, $ |G:N_{G}(HM)| $ is a $ \pi(HM/M) $-number. It follows that  $ H $ satisfies  $ \LL $-$ \Pi $-property in $ G $.
\end{proof}

If $ P $ is either an odd order $ p $-group or a quaternion-free $ 2 $-group, then we use $ \Omega(P) $ to denote the subgroup $ \Omega_{1} (P) $.  Otherwise, $ \Omega (P) = \Omega_{2} (P) $.

\begin{lemma}\label{hypercenter}
	Let $ P $ be a normal $ p $-subgroup of $ G $ and $ D $ a Thompson critical subgroup of $ P $ \cite[page 186]{Gorenstein-1980}. If $ P/\Phi(P) \leq Z_{\UU}(G/\Phi(P)) $ or  $ \Omega(D) \leq Z_{\UU}(G) $, then $ P \leq  Z_{\UU}(G) $.
\end{lemma}
\begin{proof}[\bf{Proof}]
	By \cite[Lemma 2.8]{Chen-xiaoyu-2016}, the conclusion follows.
\end{proof}

\begin{lemma}[{\cite[Lemma 2.10]{Chen-xiaoyu-2016}}]\label{critical}
	Let $ D $ be a Thompson critical subgroup of a nontrivial $ p $-group $ P $.
	
	
	{\rm (1)} If $ p > 2 $, then the exponent of $ \Omega_{1}(D) $ is $ p $.
	
	{\rm (2)} If $ P $ is an abelian $ 2 $-group, then the exponent of $ \Omega_{1}(D) $ is $ 2 $.
	
	{\rm (3)} If $ p = 2 $, then the exponent of $ \Omega_{2}(D) $ is at most $ 4 $.
\end{lemma}

\begin{lemma}[{\cite[Lemma 3.1]{Ward}}]\label{charcteristic}
	Let $ P $ be a nonabelian quaternion-free $ 2 $-group. Then $ P $ has a characteristic subgroup of index $ 2 $.
\end{lemma}

\begin{lemma}[{\cite[Lemma 2.12]{Qiu-partial}}]\label{hyper}
	Let $ N $ be a normal subgroup of $ G $ and $ P\in {\rm Syl}_{p}(N) $. Then $ N\leq Z_{\UU_{p}}(G) $ if and  only if  $ P $ satisfies the following:
	
	\begin{enumerate}[\rm(1)]
		\item All  subgroups of $ P $ of order $ p $ are contained in $ Z_{\UU_{p}}(G) $;
		
		\item If  $ P $ is not quaternion-free, we  further suppose that all cyclic subgroups of $ P $ of order $ 4 $ are contained in $ Z_{\UU_{p}}(G) $.
		
	\end{enumerate}

\end{lemma}

\begin{lemma}\label{less}
	Let $ G $ be a $ p $-soluble group of order divisible by $ p $. Then $ |G/O_{p'p}(G)|_{p} <
	|O_{p'p}(G)|_{p} $ .
\end{lemma}
\begin{proof}[\bf{Proof}]
	By \cite[Corollary 1.9(iii)]{Wolf-1984}, the conclusion holds.
\end{proof}

\begin{lemma}[{\cite[Proposition 2.5]{Qian-2012}}]\label{Qian}
	Let $ N = W_{1} \times  \cdots \times W_{s}  $ be a normal  subgroup of $ G $ with $ C_{G}(N) = 1 $, where every $ W_{i} $  $ (1\leq i\leq s) $ is a nonabelian simple group of order divisible by $ p $. Then $ |G/N|_{p} < |N|_{p}  $.
\end{lemma}

\begin{lemma}\label{Also}
	Let $ N $ be a minimal normal subgroup of  $ G $ and $ K $ a subgroup of $ G $.  Assume that $ |N|=|K|=p $. If $ KN $ satisfies  $ \LL $-$ \Pi $-property in $ G $, then $ K $ also  satisfies  $ \LL $-$ \Pi $-property in $ G $.
\end{lemma}

\begin{proof}[\bf{Proof}]
	If $ K = N $, there is nothing to prove. Hence we may assume that $ |KN|=p^{2} $. Let $ K^{G}/A $ be an arbitrary chief factor of $ G $. Assume that $ (KNA/A)\cap (K^{G}/A) $ has order $ p^{2} $. Then $ KN\cap A=1 $ and $ KNA\leq K^{G} $. But $ K^{G}=NA $ and $ |(KNA/A)\cap (K^{G}/A )| = |(KNA/A ) \cap (NA/A)| = | NA/A| = p $. This contradicts our  assumption. Hence we have $ |(KNA/A)\cap (K^{G}/A)|\leq p $. If $ |KA/A|= 1 $, then $ K\leq A $. This is impossible. Thus we can assume that $ |KA/A| = p = |(KNA/A)\cap (K^{G}/A)| $. Therefore $ KA=KNA\cap K^{G} $.  Note that  $ A $ or $ NA $ is a maximal $ G $-invariant subgroup  of $ K^{G}N $. Since $ KN $ satisfies  $ \LL $-$ \Pi $-property in $ G $, we have that  $ |G:N_{G}(KNA)| $ is a $ p $-number, and so $ |G:N_{G}(KA)|=|G:N_{G}(KNA\cap K^{G})| $  is a $ p $-number. Hence $ K $   satisfies  $ \LL $-$ \Pi $-property in $ G $.
\end{proof}

\begin{lemma}\label{invariant}
	Let $ P\in {\rm Syl}_{p}(G) $, $ d $ be a power of $ p $ with $ p\leq d\leq |P| $, and let $ N $ be a minimal normal subgroup of $ G $ with $ d\big | |N| $. Assume  that every subgroup  of $ P $ of order $ d $ satisfies  $ \LL $-$ \Pi $-property in $ G $. Then $ |N| = d  $; and if in addition
	$ d \geq  p^{2} $, then $ N $ is the unique minimal normal subgroup of $ G $ of order divisible by $ p $.
\end{lemma}
\begin{proof}[\bf{Proof}]
	Let $ H $ be a normal  subgroup of $ P $ of order $ d $ such that $ H\leq N $.  Since $ H $ satisfies  $ \LL $-$ \Pi $-property in $ G $ and $ H^{G} = N $, we have  that $ | G  : N _{G} (H)| $ is a $ p $-number. Thus $ H=N $ has order $ d $.
	
	Assume that $ d \geq p^{2} $ and $ G $ possesses a minimal normal subgroup $ T $ with $ p \big | |T| $ and  $ T \not = N $. Let  $ U $ be a normal  subgroup of $ P $ contained in  $ P\cap (T\times N) $ with  order $ d $ such that $ |U\cap T|=p $ and $ |U\cap N|=d/p $. Clearly, $ U^{G}/T=NT/T $. By hypothesis,  $ U $ satisfies  $ \LL $-$ \Pi $-property in $ G $. Hence $ |G:N_{G}(UT)| $ is a $ p $-number. Thus $ UT\unlhd G $,  contradicting the minimality of $ TN/T $ as a normal subgroup of $ G/T $.
\end{proof}

\begin{lemma}\label{p-group}
	Let $ H $ be a nonidentity $ p $-subgroup of $ G $ and $ N $ a mininal normal subgroup of $ G $. If $ H $ satisfies $ \LL $-$ \Pi $-property in $ G $ and $ H\leq N $,  then $ N $ is a $ p $-group.
\end{lemma}
\begin{proof}[\bf{Proof}]
	Clearly, $ H^{G}=N $. By hypothesis, $ H $ satisfies $ \LL $-$ \Pi $-property in $ G $. Hence  $ |G:N_{G}(H)| $ is a $ p $-number. This implies
	that $ G = N_{G}(H)P  $, where $ P $ is a Sylow $ p $-subgroup of $ G $ containing $ H $. By \cite[Lemma 3.4.9]{Guo-2000}, it follows that $ H\leq P_{G} $. Thus $ N=H^{G}\leq O_{p}(G) $.
\end{proof}

\section{Proofs}
\begin{theorem}\label{small}
	Let $ P $ be a normal $ p $-subgroup of $ G $. Then $ P\leq Z_{\UU}(G) $ if $ P $ satisfies the following:
	
	\begin{enumerate}[\rm(1)]
		\item Every  subgroup  of $ P $ of order $ p $ satisfies $ \LL $-$ \Pi $-property in $ G $;
		
		\item If  $ P $ is not quaternion-free, we  further suppose that every  cyclic subgroup of $ P $ of order $ 4 $ satisfies $ \LL $-$ \Pi $-property in $ G $.
		
	\end{enumerate}

\end{theorem}
\begin{proof}[\bf{Proof}]
	Assume that this theorem  is not true and let $ ( G, P ) $ be a counterexample for which $ | G || P | $ is minimal. We proceed via the following steps.
	
	\vskip0.1in
	
	\textbf{Step 1.} $ P $ has a unique  maximal $ G $-invariant subgroup, say $ N $. Moreover, $ N\leq Z_{\UU}(G) $ and $ P/N $ is not cyclic.

	It is easy to see that $ (G, N) $ satisfies the hypotheses of the theorem. By the choice of $ (G, P) $, we have that $ N \leq Z_{\UU}(G) $. Now assume that $ T $ is a  maximal $ G $-invariant subgroup of $P$, which is different from $ N $. Then $ T\leq Z_{\UU}(G) $  with a similar argument as above. Thus $ P=TN\leq Z_{\UU}(G) $, a contradiction. This shows that $ N $ is the unique  maximal $ G $-invariant subgroup of $ P $. If $ P/N $ is cyclic, then $ P/N \leq Z_{\UU}(G) $,  and so $ P \leq
	Z_{\UU}(G) $, which is impossible. Hence $ P/N $ is noncyclic.
	
	
	\vskip0.1in
	
	\textbf{Step 2.} The exponent of $ P $ is $ p $ or $ 4 $ (when $ P $ is  not quaternion-free).

	Let $ D $ be a Thompson critical subgroup of $ P $.
	If $ \Omega(D)< P $, then $ \Omega(D)\leq N $ by the uniqueness of $ N $. In view of Step 1, we have  $ \Omega(D)\leq Z_{\UU}(G) $. Applying Lemma \ref{hypercenter}, we know that $ P\leq Z_{\UU}(G) $, a contradiction. Thus $ P = D = \Omega (D) $. If $ P $ is a nonabelian quaternion-free $ 2 $-group, then $ P $ has a characteristic subgroup $ K $ of index $ 2 $ by Lemma \ref{charcteristic}. In view of Step 1, we have $ K=N \leq Z_{\UU}(G) $, and hence  $ P\leq Z_{\UU}(G) $, which is impossible. This means that $ P $ is a nonabelian $ 2 $-group if and only if $ P $ is not quaternion-free. By  Lemma \ref{critical}, the exponent of $ P $ is $ p $ or $ 4 $ (when $ P $  is not quaternion-free), as wanted.
	
	\vskip0.1in
	
	\textbf{Step 3.} The final contradiction.

	Let $ G_{p} $ be a Sylow $ p $-subgroup of $ G $. Note that $ P/N \cap
	Z(G_{p}/N) > 1 $. Let $ X/N $ be a subgroup of $ P/N \cap
	Z(G_{p}/N) $ of order $ p $. Then we can choose an element $ x\in X-N $. Write $ H=\langle x \rangle $. Then  $ X=HN $ and $ H $ is a subgroup of order $ p $ or $ 4 $ (when $ P $ is not quaternion-free) by Step 2. The uniqueness of $ N $ implies $ H^{G}=P $.   By hypothesis, $ H $ satisfies $ \LL $-$ \Pi $-property in $ G $. Then $ |G:N_{G}(X)|=|G:N_{G}(HN)| $ is a $ p $-number, and thus $ X\unlhd G $. This forces that $ P/N=X/N=HN/N\cong H/(H\cap N) \lesssim \langle x \rangle $, which contradicts Step 1. This final contradiction completes the proof.
\end{proof}

\begin{theorem}\label{super}
	Let  $ N $ be a normal subgroup of  $ G $  and $ P\in {\rm Syl}_{p}(N) $. Then $ N\leq Z_{\UU_{p}}(G) $ if $ P $ satisfies the following:
	
	\begin{enumerate}[\rm(1)]
		\item Every  subgroup of $ P $ of order $ p $ satisfies $ \LL $-$ \Pi $-property in $ G $;
		
		\item If $ P $ is not quaternion-free, we  further suppose that every cyclic subgroup of $ P $ of order $ 4 $ satisfies $ \LL $-$ \Pi $-property in $ G $.
		
	\end{enumerate}	
	
	

\end{theorem}
\begin{proof}[\bf{Proof}]
	Assume that this theorem is not true and let $ (G , N) $ be a minimal counterexample for which $ |G||N| $ is minimal.
	
	\vskip0.1in
	
	\textbf{Step 1.} $ O_{p'}(N)=1 $.

	Clearly the hypotheses hold for $ (G/O_{p'}(N), N/O_{p'}(N)) $ and thus the minimal choice of $ ( G , N ) $ implies that $ O_{p'}(N) = 1 $.
	
	\vskip0.1in
	
	\textbf{Step 2.} $ N $ is not $ p $-soluble.

	Assume that $ N $ is $ p $-soluble. Since $ O_{p'}(N)=1 $, it follows from \cite[Lemma 2.10]{ball-2009} that $ F^{*}_{p}(N) = F_{p}(N) =
	O_{p}(N) $.  By Theorem \ref{small}, we have $ F^{*}_{p}(N) =O_{p}(N) \leq Z_{\UU}(G) \leq Z_{\UU_{p}}(G) $. Hence $ N\leq Z_{\UU_{p}}(G) $ by \cite[Lemma 2.13]{Su-2014}, a contradiction.
	
	\vskip0.1in
	
	\textbf{Step 3.} $ N $ has a unique maximal $ G $-invariant subgroup $ U $. Moreover, $ 1<U\leq Z_{\UU_{p}}(G) $.

	Let  $ H $ be a cyclic subgroup of $ P $ of order $ p $ or $ 4 $ (when $ P $ is not quaternion-free).
	Assume that  $ N $ is a minimal normal subgroup of $ G $. Then $ H^{G}=N $. 	By Lemma \ref{p-group}, $ N $ is a $ p $-group, contrary to Step 2. Hence $ N $ has a maximal $ G $-invariant subgroup $ U\not =1 $. It is easy to see  that the hypotheses also hold for $ ( G , U ) $ and the minimal choice of $ ( G , N) $ yields that $ U\leq  Z_{\UU_{p}}(G) $. Assume that $ R $ is a maximal $ G $-invariant subgroup of $ N $, which is different from $ U $. Then $ R\leq Z_{\UU_{p}}(G) $, and thus $ N=UR\leq Z_{\UU_{p}}(G) $, a contradiction. Thus $ U $ is the unique maximal $ G $-invariant subgroup of $ N $.
	
	\vskip0.1in
	
	\textbf{Step 4.}  Every cyclic subgroup of $ P $ of order $ p $ or $ 4 $ (when $ P $ is not quaternion-free)  is contained in $ U $.

	Assume that there exists a cyclic subgroup $ L $ of $ P $ of order $ p $ or $ 4 $ (when $ P $ is not quaternion-free)  such that $ L\nleq U $.  By hypothesis, $ L $ satisfies $ \LL $-$ \Pi $-property in $ G $. Applying  Lemma \ref{ove}(2),  $ LU/U $ satisfies $ \LL $-$ \Pi $-property in $ G/U $. The uniqueness of  $ U $ yields that $ L^{G}=N $, and so $ (LU/U)^{G/U}=N/U $. By Lemma \ref{p-group}, it turns out that $ N / U $ is a $ p $-group.  By Step 3, we know that $ N $ is $ p $-soluble, contrary to Step 2.   Therefore, Step 4 holds.
	
	\vskip0.1in
	
	\textbf{Step 5.}  The final contradiction.

	Since $ U\leq Z_{\UU_{p}}(G) $, it follows from Step 4 and Lemma \ref{hyper} that $ N\leq Z_{\UU_{p}}(G) $. This  final contradiction completes the proof.
\end{proof}

\begin{lemma}\label{little}
	Let $ P\in {\rm Syl}_{p}(G) $ and $ d $ be a power of $ p $ with $ p^{2}\leq d<|P| $. Assume that every    subgroup of  $ P $ of order $ d $ satisfies  $ \LL $-$ \Pi $-property in $ G $ and
	$ G $ possesses a minimal normal  subgroup $ N $ with $ |N| = d/p $. Then $ G/N $ is $ p $-supersoluble.
\end{lemma}
\begin{proof}[\bf{Proof}]
	By  Lemma \ref{ove}(2) and induction, we may assume that  $ O_{p'}(G) = 1 $. Let $ A/N $ be a subgroup of $ P/N $ of order $ p $. Then $ A $ has order $ d $. By hypothesis, $ A $ satisfies  $ \LL $-$ \Pi $-property in $ G $. Then $ A/N $ satisfies  $ \LL $-$ \Pi $-property in $ G/N $. Therefore,  every  subgroup  of $ P/N $ of order $ p $ satisfies  $ \LL $-$ \Pi $-property in $ G/N $.
	
	If $ p > 2 $, or $ p = 2 $ and $ P/N $  is quaternion-free, then by Theorem \ref{super},  $ G/N $ is $ p $-supersoluble, as desired. Hence  we may assume that $ p=2 $ and $ P/N $  is not quaternion-free.
	Let $ X/N $ be an arbitrary cyclic subgroup of $ P/N $ of order $ 4 $. Assume that $ N \leq \Phi(X) $. Then $ X $ is cyclic and thus $ |N| = 2 $ and $ d = 4 $. By hypothesis,  every subgroup of $ P $ of order $ 4 $ satisfies  $ \LL $-$ \Pi $-property in $ G $. Now we claim that  every subgroup $ Y $ of $ P $ of order
	$ 2 $ satisfies  $ \LL $-$ \Pi $-property in $ G $.  It is no loss to assume that $ N\not =Y $. Then $ YN $ has order $ 4 $. By  hypothesis,  $ YN $ satisfies  $ \LL $-$ \Pi $-property in $ G $. Applying Lemma  \ref{Also}, we deduce that $ Y $ satisfies $ \LL $-$ \Pi $-property in $ G $, as claimed. It follows from Theorem  \ref{super} that $ G $ is $ p $-supersoluble. Thus $ G/N $ is $ p $-supersoluble.
	
	Hence we may suppose that $ N \nleq  \Phi(X) $.  Then there exists a maximal subgroup $ X_{1} $ of $ X $ such that $ X = X_{1}N $. Notice that $ |X_{1}|=2|N|=d $. Then $ X_{1} $ satisfies $ \LL $-$ \Pi $-property in $ G $. By Lemma \ref{ove}(2), $ X/N = X_{1}N/N $ satisfies $ \LL $-$ \Pi $-property in $ G/N $. This means that every cyclic subgroup of $ P/N $ of order $ 4 $ satisfies $ \LL $-$ \Pi $-property in $ G/N $. Note that every subgroup of $ P/N $ of order $ p=2 $ satisfies  $ \LL $-$ \Pi $-property in $ G/N $.  According to Theorem \ref{super}, $ G/N $ is $ p $-supersoluble. Hence the proof of the lemma is complete.
\end{proof}

\begin{lemma}\label{G/N}
	Let $ P\in {\rm Syl}_{p}(G) $ and $ d $ be a power of $ p $ with $ p^{2}\leq d< |P| $. Assume that every    subgroup of  $ P $ of order $ d $ satisfies  $ \LL $-$ \Pi $-property in $ G $ and  $ G $ possesses a minimal normal subgroup $ N $ of order $ d $. Then  $ G/N $ is $ p $-supersoluble.
\end{lemma}
\begin{proof}[\bf{Proof}]
	Suppose that the result  is false and $ G $ is a counterexample of minimal order.
	By Lemma \ref{ove}(2), the hypotheses are inherited by  $ G/O_{p'}(G) $, so we can assume that $ O_{p'}(G) = 1 $. By Lemma \ref{invariant},  $ N $ is the unique normal subgroup of $ G $.
	
	At first, we will show that every subgroup of $ P/N $ of order $ p $ satisfies  $ \LL $-$ \Pi $-property in $ G/N $.  Let $ X/N $ be an arbitrary  subgroup of $ P/N $ of order $ p $. Then $ X $ is not a cyclic subgroup
	since $ N $ is not cyclic. It follows that $ X $ has another subgroup $ X_{1} $ of order $ d $ such that $ X = X_{1}N $. By hypothesis, $ X_{1} $ satisfies  $ \LL $-$ \Pi $-property in $ G $. Applying Lemma \ref{ove}(2), $ X/N=X_{1}N/N $ satisfies  $ \LL $-$ \Pi $-property in $ G/N $. Hence  every subgroup of $ P/N $ of order $ p $ satisfies  $ \LL $-$ \Pi $-property in $ G/N $.
	
	If $ p>2 $ or $ p=2 $ and $ P/N $ is quaternion-free, then $ G/N $ is $ p $-supersoluble according to Theorem \ref{super}. Therefore we may suppose that $ p=2 $ and $ P/N $ is not quaternion-free.
	
	Let $ T \unlhd G $ be minimal  such that $ N < T < G $ and $ T/N \nleq  Z_{\UU_{2}}(G/N) $. Obviously, $ T/N $ possesses a unique maximal $ G/N $-invariant subgroup, say $ U/N $. The uniqueness of $ N $ yields that $ U $ is the unique maximal $ G $-invariant subgroup of $ T $. Notice that $ U/N=T/N\cap Z_{\UU_{2}}(G/N) $. By \cite[Chapter 2, Lemma 5.25]{Guo2015}, all  $ 2 $-supersoluble groups  are  $ 2 $-nilpotent. Therefore,  $ U/N $ is $ 2 $-nilpotent.  By Lemma \ref{hyper}, $ (P\cap T)/N $ possesses a cyclic subgroup $ \langle a \rangle N/N $ of order $ 2 $ or $ 4 $  such that $ a \in T -U $.  Let $ a_{0} \in T-U  $ be a $ p $-element such that $ o(a_{0})  $  is as small as possible. Since $ o(a_{0}) \leq o(a) $ and $ N $ is elementary abelian, we have that  $ o(a_{0})\leq 8 $. The minimality of $ o(a_{0}) $ implies that $ a_{0}^{2}\in U $.
	
	We shall complete the proof in the following two cases.
	
	\vskip0.1in
	\textbf{Case 1.} $ o(a_{0})> d $.

	In this case, $ o(a_{0}) = 8 $ and $ |N| = d=4 $.  Suppose that $ N \nleq  Z(T) $.  The uniqueness of $ U $ yields $ C_{T}(N) \leq U $.  Therefore $ T/C_{T}(N)\lesssim {\rm Aut}(N)\cong S_{3} $. It follows that $ T/U\leq Z_{\UU_{2}}(G/U) $, and thus $ T/N\leq Z_{\UU_{2}}(G/N) $, a contradiction. Hence $ N \leq  Z(T) $.  Observe that $ U/N $ is $ 2 $-nilpotent and  $ N\leq Z(T)\cap U\leq Z(U) $.  The uniqueness of $ N $ implies that $ O_{2'}(U) = 1 $ and $ U $ is a  $ 2 $-group. Since $ a_{0}^{2}\in U $ and $ o(a_{0}^{2})=4 $, we have $ \Phi(U)>1 $. Then $ N\leq \Phi(U) $ by the uniqueness of $ N $. Since $ U/N\leq Z_{\UU_{2}}(G/N) $, it follows from \cite[Lemma 2.10]{Su-2014} that $ U\leq Z_{\UU_{2}}(G) $. This contradicts the fact that $ |N|=4 $.
	
	\vskip0.1in
	\textbf{Case 2.} $ o(a_{0})\leq d $.

	Let $ H $ be a subgroup of $ P $ of  order $ d $ such that $ \langle a_{0} \rangle\leq H\leq \langle a_{0} \rangle N $.  Since
	$ N $ is the unique minimal normal subgroup of $ G $, we have $ H^{G} \geq N $. Since $ a_{0}\not \in U $ and
	$ U $ is the unique maximal $ G $-invariant subgroup of $ T $, we have that $ H^{G} = T $. By hypothesis, $ H $ satisfies $ \LL $-$ \Pi $-property in $ G $, and  hence $ |G:N_{G}(HU)| $ is a $ 2 $-number. Thus $ G=N_{G}(HU)P $. It follows that $ T=H^{G}=(HU)^{G}=H^{P}U $, and so $ T/U $ is a $ 2 $-group.
	
	Note that  $ U/N $ has a normal $ 2 $-complement, say $ R/N $. Write $ \overline{G}=G/R $. We will show that the exponent of $ \overline{T} $ is $ 2 $ or $ 4 $ (when $ \overline{T} $  is not quaternion-free).   Let $ \overline{D} $ be a Thompson critical subgroup of $ \overline{T} $.  If $ \Omega(\overline{T})< \overline{T} $, then $ \Omega(\overline{T})\leq \overline{U} $ by the  uniqueness of $ U $. Thus  $ \Omega(\overline{T})\leq Z_{\UU}(\overline{G}) $.
	By Lemma \ref{hypercenter}, we have that $ \overline{T}\leq Z_{\UU}(\overline{G}) $, a  contradiction. Hence  $ \overline{T} = \overline{D} = \Omega(\overline{D}) $.  If $ \overline{T} $ is a nonabelian quaternion-free $ 2 $-group, then $ \overline{T} $ has a characteristic subgroup $ \overline{T_{0}} $ of index $ 2 $ by Lemma \ref{charcteristic}. The uniqueness of $ U $ implies $ T_{0}=U $. Thus  $ \overline{T} \leq Z_{\UU}(\overline{G}) $, which is impossible. This means that $ \overline{T} $ is a nonabelian $ 2 $-group if and only if $ \overline{T} $ is not quaternion-free.  In view of  Lemma \ref{critical}, the exponent of $ \overline{T} $ is $ 2 $ or $ 4 $ (when $ \overline{T} $  is not quaternion-free), as desired.

	Since $ N $ is elementary abelian, we see that every element of $ P\cap T $ has order at most $ 8 $. Clearly, $ |T/U|\geq 4 $. There exists  a normal subgroup $ Q $ of $ P $  such that $ P\cap U< Q< P\cap T $ and $ |Q:(P\cap U)|=2 $.  Pick a nonidentity  element $ b\in Q-(P\cap U) $. Then $ o(b)\leq 8 $ and $ b^{2} \in (P\cap U) $.

	If $ o(b) >d $, then  $ o(b) = 8 $ and $ |N|=d=4 $. With a similar argument as Case 1, we have  $ U\leq Z_{\UU_{2}}(G) $. This contradicts the fact that $ |N|=4 $.

	Assume that $ o(b) \leq d $. Let $ S $ be a  subgroup of $ Q $ of  order  $ d $ such that $ \langle b \rangle \leq S $. If $ b\in U $, then  $ P\cap U<\langle b \rangle(P\cap U)\leq U $, a contradiction. Therefore $ b
	\not \in U $. The  uniqueness of $ U $ implies that $ S^{G}=T $, and so $ Q=SU $. By hypothesis, $ S $ satisfies  $ \LL  $-$ \Pi $-property in $ G $. Then  $ |G:N_{G}(Q)|=|G:N_{G}(SU)| $ is a $ 2 $-number. Thus $ Q\unlhd G $, which contradicts the fact that $ U $ is the  unique maximal $ G $-invariant subgroup of $ T $. Our proof is now complete.
\end{proof}

\begin{theorem}\label{geq}
	Let $ P $ be a Sylow $ p $-subgroup of $ G $, and let $ d $ be a power of $ p $ such that $ p^{2} \leq d \leq  |P\cap O_{p'p}(G)|/p $. If every subgroup of  $ P $ of order $ d $ satisfies  $ \LL $-$ \Pi $-property in $ G $,  then $ G $ is $ p $-supersoluble.
\end{theorem}
\begin{proof}[\bf{Proof}]
	We proceed by induction on $ |G| $. By Lemma \ref{ove}(2), we can assume that $ O_{p'}(G) = 1 $. Then  $ O_{p'p}(G)=O_{p}(G) $ and $ |O_{p}(G)|\geq dp $. Let $ N $ be a minimal normal subgroup of $ G $ contained in $ O_{p}(G) $.  By Lemma \ref{invariant}, $ |N|\leq d $.
	
	Next we claim that $ G/N $ is $ p $-supersoluble. If $ |N| = d/p  $ or $ |N|=d $, then by  Lemmas \ref{little} and \ref{G/N},  $ G/N $ is $ p $-supersoluble.
	Assume that $ |N| \leq d/p^{2} $, that is , $ d/|N|\geq p^{2} $.  Clearly, $ |O_{p'p}(G/N)|\geq dp/|N| $. Applying  Lemma \ref{ove}(2), the hypotheses are inherited by $ G/N $.  Then $ G/N $ is $ p $-supersoluble by induction, and the claim holds.

	Assume that $ \Phi(G) \cap P > 1 $. Let $ N_{0} $ be a minimal normal  subgroup of $ G $ contained in  $ \Phi(G) \cap P $. The
	preceding claim shows the $ p $-supersolubility of $ G/N_{0} $. Now $ G $ is $ p $-supersoluble by \cite[Lemma 2.10]{Su-2014}.
	
	Assume that $ \Phi(G)\cap P=1 $. Then $ P $ is a direct product of some minimal normal subgroups of $ G $ according to \cite[Lemma 2.13]{Skiba-2007}. Since $ P $ is not a minimal normal subgroup of $ G $, it possesses different minimal $ G $-invariant subgroups, say $ N_{1} $, $ N_{2}  $. Since $ G/N_{1} $ and $ G/N_{2} $ are $ p $-supersoluble by the claim, we conclude that $ G $ is $ p $-supersoluble. The proof is complete.
\end{proof}

\begin{theorem}\label{sqrt}
	Let $ P $ be a Sylow $ p $-subgroup of $ G $, and let $ d $ be a power of $ p $ such that $ p^{2}\leq d \leq \sqrt{|P|} $. If every    subgroup of  $ P $ of order $ d $ satisfies  $ \LL $-$ \Pi $-property in $ G $, then $ G $ is $ p $-supersoluble.
\end{theorem}
\begin{proof}[\bf{Proof}]
	We proceed by induction on $ |G| $. It is no loss to assume that $ O_{p'}(G)=1 $.
	
	Assume  that all minimal normal subgroups of $ G $ are nonabelian. Set $ T=\mathrm{Soc}(G) $. Then $ T=V_{1}\times\cdots \times V_{s} $, where every $ V_{i} $ $ (1\leq  i\leq s) $ is a minimal normal subgroup of $ G $.  Clearly, $ C_{G}(T) = 1 $. Since $ O_{p'}(G)=1 $, it follows from Lemma \ref{Qian} that $ |T|_{p}>\sqrt{|P|}\geq d $. Let $ L $ be a subgroup of $ P\cap T $ of order $ d $ such that $ L\unlhd P $. Choose  a maximal $ G $-invariant subgroup $ K $ of $ L^{G} $. Then $ |G:N_{G}(LK)| $ is a $ p $-number. Thus $ LK\unlhd G $, and $ L^{G}=LK $.
	This implies that $ L^{G}/K $ is a $ p $-group, contrary to the fact that $ L^{G}/K $ is  isomorphic to some $ V_{i} $. Therefore $ G $ possesses a minimal normal $ p $-subgroup, say $ N $. By Lemma \ref{invariant}, $ |N|\leq d $.
	
	Next we claim that $ G/N $ is $ p $-supersoluble. If $ |N| = d/p  $ or $ |N|=d $, then by  Lemmas \ref{little} and \ref{G/N},  $ G/N $ is $ p $-supersoluble.
	Assume that $ |N| \leq d/p^{2} $. Then $ p^{2} \leq d/|N|\leq \sqrt{|P/N|} $. Applying  Lemma \ref{ove}(2), the hypotheses are inherited by $ G/N $. By induction,  $ G/N $ is $ p $-supersoluble, as desired.

	Consequently, $ G $ is $ p $-soluble. Then $ |O_{p'p}(G)|_{p}>\sqrt{|P|}\geq
	d  $ by Lemma \ref{less}. Applying Theorem \ref{geq}, we get  that $ G $ is $ p $-supersoluble.
\end{proof}

\begin{proof}[\bf{Proof of Theorem \ref{main}}]
	By Theorems \ref{super}, \ref{geq} and \ref{sqrt}, the conclusion follows.
\end{proof}



\section*{Acknowledgments}

This work is supported by the National Natural Science Foundation of China (Grant No.12071376).



\small


\end{document}